\documentclass[10pt]{article}
\usepackage{amssymb}
\usepackage{amsfonts}
\usepackage{amsthm}
\usepackage{amsmath}
\usepackage{fullpage}

\newtheorem{Lemma}{Lemma}

\newtheorem{Theorem}{Theorem}

\newtheorem{Definition}{Definition}

\newcommand{\Z}{\mathbb{Z}}

\newcommand{\C}{\mathbb{C}} 
\newcommand{\Q}{\mathbb{Q}} 

\usepackage[pdftex]{graphicx}
\graphicspath{{C:/Users/Sam/Desktop/PaperFigures/}}
\DeclareGraphicsExtensions{.pdf,.png}

\begin{document}
\title{On the Probability of Relative Primality in the Gaussian Integers}
\author{Bianca De Sanctis and Samuel Reid}
\maketitle

\begin{abstract}
This paper studies the interplay between probability, number theory, and geometry in the context of relatively prime integers in the ring of integers of a number field. In particular, probabilistic ideas are coupled together with integer lattices and the theory of zeta functions over number fields in order to show that
$$P(\gcd(z_{1},z_{2})=1) = \frac{1}{\zeta_{\Q(i)}(2)}$$
where $z_{1},z_{2} \in \mathbb{Z}[i]$ are randomly chosen and $\zeta_{\Q(i)}(s)$ is the Dedekind zeta function over the Gaussian integers. Our proof outlines a lattice-theoretic approach to proving the generalization of this theorem to arbitrary number fields that are principal ideal domains.
\end{abstract}

\section{Introduction}
Number theory is historically defined as the study of the integers and is one of the oldest fields of mathematics to be developed with sophistication. Throughout antiquity, multiple mutually exclusive civilizations developed basic mathematical notions regarding shape and number. By the $2^{nd}$ century BC, ancient Egyptian and Babylonian mathematicians were able to find the positive real roots of equations of the form $x^2 + bx = c$ through arithmetical and geometrical procedures. During the $1^{st}$ century BC in India, the concept of zero, negative numbers and the decimal number system were developed; and during the $1^{st}$ century AD in China, methods for approximating $\pi$ and performing the division algorithm were developed. Slightly later in Greece, Books VII to IX of Euclid's \textit{Elements} defined prime numbers, divisibility, and gave the first well-documented study of relative primality. Extensions of Euclid's work were developed during the Islamic Golden Age that took place from the $9^{th}$ to $12^{th}$ century and much of the remaining knowledge from Babylonian, Egyptian, and Greek mathematics was compiled by Arab scholars \cite{Berggren}. Translations of these writings subsequently spread to Europe during the Middle Ages and provided European mathematicians the foundation to work for centuries on contributing to the growing body of knowledge of mathematics.

As a result of the work of generations of mathematicians in Asian, Arab, and European cultures, research regarding the Riemann zeta function $$\zeta(s) = \sum_{n=1}^{\infty} \frac{1}{n^{s}}$$ was first published by Leonhard Euler in the $18^{th}$ century. In particular, Euler proved that
$$\zeta(s) = \prod_{p \; \text{prime}} \left(\frac{1}{1-p^{-s}}\right)$$
which implies that there are infinitely many primes; this form of $\zeta(s)$ is known as the Euler product and is frequently the form of $\zeta(s)$ that is used in the context of number theory. This important function of a complex variable $s$ has been seen to be ubiquitous in modern mathematics, making appearances in higher-dimensional sphere packing problems \cite{Ball01011992}, the Zipf-Mandelbrot law in probability theory, the Casimir effect, and other disparate areas of mathematics and physics. Unexpectedly, $\zeta(s)$ appears in the probabilistic study of relatively prime integers, as we have the following theorem originally proved by Ernesto Ces\'{a}ro in 1881.

\begin{Theorem}\label{theorem1}
Let $x,y \in \mathbb{Z}$ be randomly chosen. Then,
$$P(\gcd(x,y)=1) = \frac{1}{\zeta(2)}$$
\end{Theorem}
\begin{proof}
We have that $P(\gcd(x,y)=1) = P(p \nmid x \cap p \nmid y, \forall p)$, where $p$ is prime. For a fixed prime $p$, the probability that $p \mid x$ and $p \mid y$ is $\frac{1}{p^2}$ and so we have that the probability this does not occur is $1 - \frac{1}{p^2}$. Therefore, taking the product over all primes we have that
$$P(\gcd(x,y)=1) = \prod_{p \; \text{prime}} \left(1 - \frac{1}{p^2}\right) = \left(\prod_{p \; \text{prime}} \left(\frac{1}{1-p^{-2}}\right) \right)^{-1} = \frac{1}{\zeta(2)}$$
\end{proof}

We remark that Euler also proved that $\zeta(2) = \frac{\pi^2}{6}$ as the solution to the Basel problem, which gives us that $P(\gcd(x,y)=1) = \frac{6}{\pi^2}$. This theorem can be generalized, in addition to rigorously defining probability by natural density as is done in \cite{Nymann1972469}, to prove the following theorem.

\begin{Theorem}\label{theorem2}
Let $x_{1},...,x_{k} \in \mathbb{Z}$ be randomly chosen. Then,
$$P(\gcd(x_{1},...,x_{k})=1) = \frac{1}{\zeta(k)}$$
\end{Theorem}
\begin{proof}
We have that $$P(\gcd(x_{1},...,x_{k})=1) = P\left(\bigcap_{i=1}^{n} p \nmid x_{i}, \forall p \right)$$ where $p$ is prime. Then for a fixed prime $p$, $$P\left(\bigcap_{i=1}^{n} p \mid x_{i}, \forall p \right) = \frac{1}{p^{k}}$$ and so we have that the probability this does not occur is $1 - \frac{1}{p^k}$. Therefore, taking the product over all primes we have that
$$P(\gcd(x_{1},...,x_{k})=1) = \prod_{p \; \text{prime}} \left(1 - \frac{1}{p^k}\right) = \left(\prod_{p \; \text{prime}} \left(\frac{1}{1-p^k}\right) \right)^{-1} = \frac{1}{\zeta(k)}$$
\end{proof}

We notice that we can consider the notions of divisibility, prime numbers, and thus relatively prime integers, in any ring $R$ which is a unique factorization domain.

\section{Relative Primality in the Gaussian Integers}
The Gaussian integers $\mathbb{Z}[i] = \{a + bi \; | \; a,b \in \mathbb{Z}\} \hookrightarrow \mathbb{C}$ are a Euclidean domain and thus a unique factorization domain. As such, we can define divisibility by $a + ib \; | \; x + iy$ if there exists $c + id \in \mathbb{Z}[i]$ such that $x + iy = (a + ib)(c + id) = (ac - bd) + i(bc + ad)$. We can now consider the problem of determining the probability of coprimality of Gaussian integers by using the extended definitions of $\gcd$ and primality in $\Z[i]$. We now cite a theorem which will aid the solution to this problem.

\begin{Theorem}\label{theorem3}
Let $z=a+bi \in \Z[i]$, then
\begin{enumerate}
\item If $a\ne 0, b\ne 0$, then $a+bi$ is prime in $\Z[i]$ iff either $a^2+b^2\equiv 1 \mod 4 $ and $a^2+b^2$ is prime in $\Z$ or $a=\pm1, b=\pm1$.
\item If $a=0, b\ne 0$, then $bi$ is prime in $\Z[i]$ iff $|b|$ is a prime in $\mathbb{Z}$ and $|b| \equiv 3 \mod 4$.
\item If $a\ne 0, b=0$, then $a$ is prime in $\Z[i]$ iff $|a|$ is a prime in $\mathbb{Z}$ and $|a| \equiv 3 \mod 4$.
\end{enumerate}
\end{Theorem}


By using this theorem, we can characterize all Gaussian primes in terms of primes in $\mathbb{Z}$. First we will investigate Gaussian primes with $a=0$ or $b=0$ as in Case 2 and Case 3 of Theorem \ref{theorem3}. In these cases, what is the probability that a fixed $a+bi$ divides a random Gaussian integer $x+iy$? The following two lemmas provide and answer to this question, which implies from Theorem \ref{theorem3} that we have a characterization of Gaussian primes with $a=0$ or $b=0$ in terms of primes $p \equiv 3 \mod{4}$ in $\mathbb{Z}$.

\begin{Lemma}\label{lemma1}
Let $z_{1},z_{2} \in \Z[i]$ and let $a + 0i$ be prime in $\Z[i]$. Then,
$$\overline{P(a \; | \; z_{1} \; \cap \; a \; | \; z_{2})} = 1 - \frac{1}{a^4}$$
\end{Lemma}
\begin{proof}
We have for $z=x+iy \in \Z[i]$ that $a + 0i \; | \; x + iy$ if and only if $a \; | \; x$ and $a \; | \; y$. We also have from the proof of Theorem \ref{theorem1} that
$$P(a \; | \; x + iy) = P(a \; | \; x \; \cap \; a \; | \; y) = \frac{1}{a^2}$$
Since $z_{1}$ and $z_{2}$ are statistically independent, this implies that
$$P(a \; | \; z_{1} \; \cap \; a \; | \; z_{2}) = P(a \; | \; z_{1}) P(a \; | \; z_{2}) = \frac{1}{a^4}$$
Therefore,
$$\overline{P(a \; | \; z_{1} \; \cap \; a \; | \; z_{2})} = 1 - \frac{1}{a^4}$$
\end{proof}

So, the probability that $a$ does not divide two random Gaussian integers is $1- \frac{1}{a^4}$ and the case for $0 + bi$ is identical. Next we must consider Gaussian primes with $a,b \neq 0$ as in Case 1 of Theorem \ref{theorem3}. In this case, what is the probability that $a+bi$ divides a random $x+iy$? The answer to this question requires a more complicated argument which uses the multiplicity lattice $\Lambda(a+bi)$ generated by $a+bi$. 

\begin{Definition}
The multiplicity lattice $\Lambda(z) \subseteq \Z \times i \Z \hookrightarrow \mathbb{C}$ of $a+bi \in \mathbb{Z}[i]$ is the sub-lattice of $\Z \times i \Z$ with a basis given by $\{a+bi,i(a+bi)\}$.
\end{Definition}

In particular, we consider the ratio of the points that fall on the multiplicity lattice of $a+bi$ to the total points in order to obtain the probability that a prime $a+bi$ with $a,b \neq 0$ divides a random $x+iy$. To do this we must introduce a new definition.

\begin{figure}[h!]
\begin{center}
\includegraphics[scale=0.39]{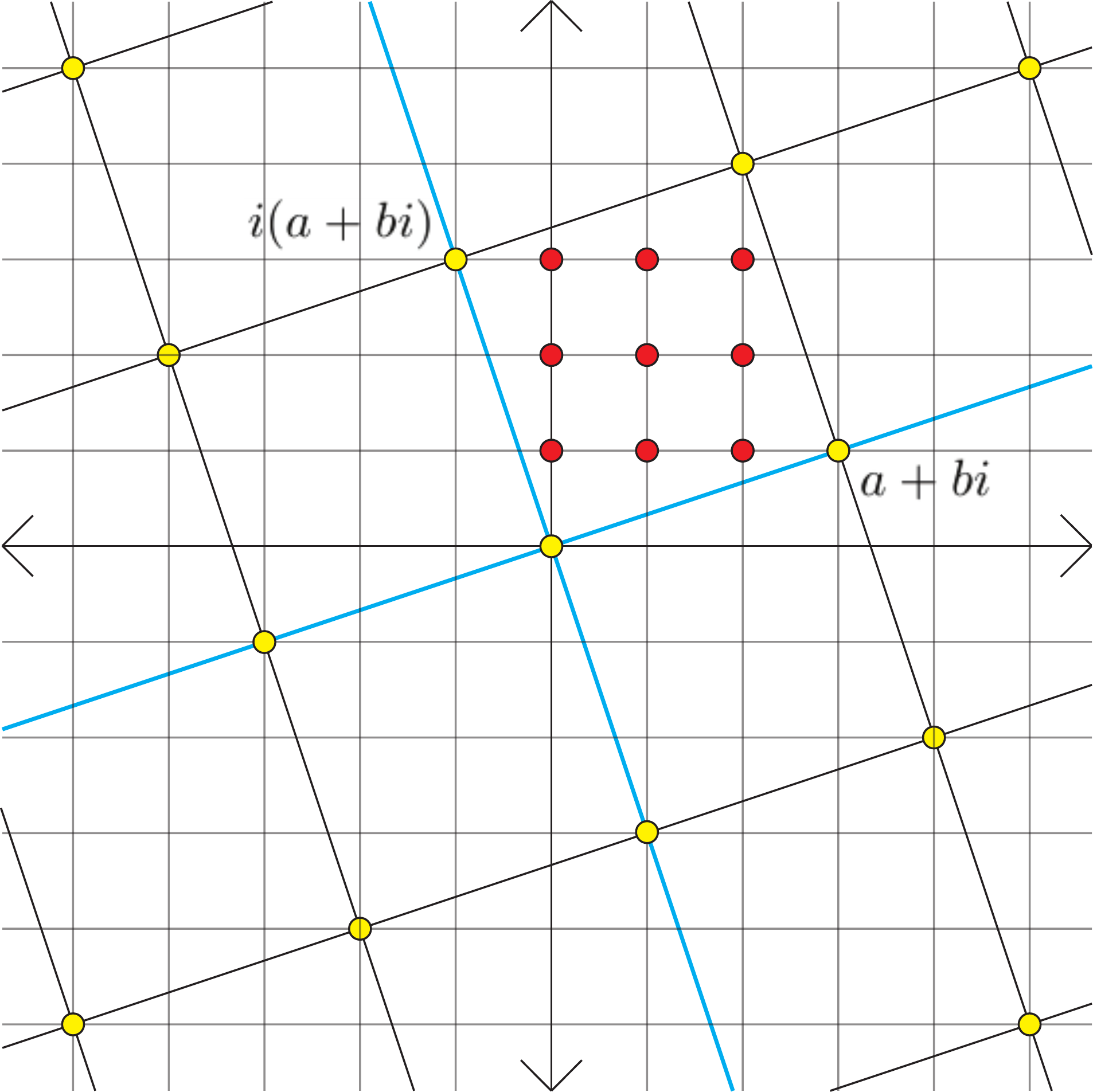}
\caption{Elements of the multiplicity lattice $\Lambda(a+bi) \hookrightarrow \C$ are denoted by yellow points, interior points of the fundamental domain with the origin as the base vertex are denoted by red points, and the blue axis represents strictly real and imaginary multiples of $a+bi$, respectively.}
\end{center}
\end{figure}

\begin{Definition}
The fundamental domain $\Gamma \subset \Z \times i \Z$ of $\Lambda(a+bi)$ with the origin as the base vertex is the square with side vectors given by $a+bi$ and $i(a+bi)$. If $x+iy \in \Lambda(a+bi)$, the fundamental domain $\Gamma_{x+iy}$ with base vertex $x+iy$ is an affine translate of $\Gamma$ from the origin to $x+iy$.
\end{Definition}

A fundamental domain of $\Lambda(a+bi)$ is always a square, and is thus a simple polygon for which a standard result in geometry known as Pick's Theorem, originally proved in 1899 by Georg Pick \cite{47270}, applies.

\begin{Theorem}
\emph{(Pick's Theorem)\\
}Let $P$ be a simple polygon on an integer lattice. Then,
$$A(P) = I(P) + \frac{B(P)}{2} - 1$$
where $A(P)$ is the area of $P$, $I(P)$ is the number of interior lattice points of $P$, and $B(P)$ is the number of boundary lattice points of $P$.
\end{Theorem}

The following lemma allows us to apply Pick's Theorem to count the interior lattice points of any fundamental domain $\Gamma_{x+iy}$ of  $\Lambda(a+bi)$.

\begin{Lemma}\label{lemma2}
The line segment containing the lattice points $x_{1} + iy_{1}$ and $x_{2} + iy_{2}$ contains no other lattice points if $\gcd(x_{2} - x_{1},y_{2} - y_{1}) = 1$.
\end{Lemma}
\begin{proof}
See the proof of Lemma 5.7 in \cite{Sally}.
\end{proof}

Combining these results, we are able to handle the Gaussian primes in Case 1 of Theorem \ref{theorem1} as follows.

\begin{Lemma}\label{lemma3}
Let $z_{1},z_{2} \in \Z[i]$ and let $a + bi$ be prime in $\Z[i]$ with $a,b \neq 0$. Then,
$$\overline{P(a + bi \; | \; z_{1} \; \cap \; a + bi \; | \; z_{2})} = 1 - \frac{1}{(a^2 + b^2)^2}$$
\end{Lemma}
\begin{proof}
We have that $a+bi \; | \; z_{1}$ and $a + bi \; | \; z_{2}$ if and only if $z_{1},z_{2} \in \Lambda(a+bi)$. We now use Pick's Theorem on an arbitrary fundamental domain $\Gamma_{x+iy}$ of $\Lambda(a+bi)$. Since $\Gamma_{x+iy}$ is an affine translate of the fundamental domain $\Gamma$, we have by Lemma \ref{lemma2} that $B(\Gamma_{x+iy}) = 4$ since $\gcd(a,b)=1$ because $a+bi$ is prime. Furthermore, $A(\Gamma_{x+iy}) = a^2 + b^2$ since $\Gamma_{x+iy}$ is a square with side lengths given by $\sqrt{a^2 + b^2}$. Thus, by Pick's Theorem we have that $I(\Gamma_{x+iy}) = a^2 + b^2 - \frac{4}{2} + 1 = a^2 + b^2 - 1$. 

Since $z_{1} \in \Z[i]$ we have that $z_{1} \in \Gamma_{x+iy}$ for some base vertex $x+iy \in \Lambda(a+bi)$. Then either $z_{1}$ is an interior point of $\Gamma_{x+iy}$ or $z_{1} = x+iy$. Since there are $I(\Gamma_{x+iy}) + 1 = a^2 + b^2$ possibilities for $z_{1}$, the probability that $z_{1}$ is the base vertex of $\Gamma_{x+iy}$ is $\frac{1}{a^2 + b^2}$. This similarly occurs for $z_{2}$ and thus,
$$P(a + bi \; | \; z_{1} \; \cap \; a + bi \; | \; z_{2}) = P(z_{1},z_{2} \in \Lambda(a+bi)) = P(z_{1} \in \Lambda(a+bi)) P(z_{2} \in \Lambda(a+bi)) = \frac{1}{(a^2 + b^2)^2}$$
Therefore,
$$\overline{P(a + bi \; | \; z_{1} \; \cap \; a + bi \; | \; z_{2})} = 1 - \frac{1}{(a^2 + b^2)^2}$$
\end{proof}

We first remark that any odd prime number $p$ can be expressed as $p = x^2 + y^2$, where $x,y \in \Z$, if and only if $p \equiv 1 \mod{4}$. This result is known as Fermat's theorem on sums of two squares and we can use it to characterize the Gaussian primes in Case 1 of Theorem \ref{theorem1}. Cumulatively, we have that Lemma \ref{lemma1} characterizes the inert Gaussian primes, and Lemma \ref{lemma3} characterizes the split and ramified Gaussian primes. Now that we have all the individual pieces, we can combine them to find the probability that $gcd(z_1,z_2)=1$ for randomly chosen $z_1,z_2\in \Z[i]$.

\begin{Theorem}\label{theorem4}
Let $z_{1},z_{2} \in \Z[i]$ be randomly chosen. Then,
$$P(\gcd(z_{1},z_{2})=1) = \frac{1}{\zeta_{\Q(i)}(2)}$$
where $\zeta_{\mathbb{Z}[i]}(s)$ is the Dedekind zeta function over the Gaussian integers.
\end{Theorem}

\begin{proof}
We have that
\begin{align*}
P(\gcd(z_{1},z_{2}) = 1) &= \overline{P(a+bi \; | \; z_{1} \; \cap \; a + bi \; | \; z_{2})}=\prod_{a+bi \; \text{prime}} \overline{P(a+bi \; | \; z_{1} \cap a+bi \; | \; z_{2})}
\end{align*}
We now decompose this product into the cases mentioned in Theorem \ref{theorem1}. If $a+bi$ is a Gaussian prime with $a=0$ or $b=0$, then $p=|a+bi|$ for some $p \equiv 3 \mod{4}$ in $\Z$, and so by Lemma \ref{lemma1} the probability that $a+bi$ does not divide $z_{1}$ and $z_{2}$ is $1 - \frac{1}{p^4}$. If $a+bi$ is a Gaussian prime with $a,b \neq 0$, then $p = |a+bi| = a^2 + b^2$ for some $p \equiv 1 \mod{4}$. This means we can start replacing $a^2+b^2$ with $p$ and span over all the primes $p \equiv 1 \mod 4$ in $\Z$. But we need to be very careful when doing this, for $\frac{1}{a^2+b^2}$ is not just the probability that $a+bi$ divides a random $x+yi$, but also the probability that $a-bi$,$-a+bi$, and $-a-bi$ divide $x+yi$. We need to figure out how many distinct cases this is, i.e. how many times we need to span over primes of the form $a^2+b^2$. Well, clearly if $a+bi$ divides some Gaussian integer, then $-a-bi$ does too. Likewise, if $a-bi$ divides a Gaussian integer, then $-a+bi$ does too. So, we are left to consider when divisibility by $a+bi$ and $a-bi$ is equivalent, that is for when a prime is ramified. By Lemma \ref{lemma3}, the probability that $a+bi$ does not divide $z_{1}$ and $z_{2}$ is $1 - \frac{1}{(a^2 + b^2)^2}$. If $p \equiv 1 \mod{4}$ then $p$ is a split prime and there are two Gaussian integers whose distinct probabilities need to be considered, and $p=2$ is the only ramified prime for which there are two Gaussian integers, namely $1+i$ and $1-i$, whose probabilities need to only be considered once. Taking the product over all primes in each case we have the following sequence of equalities proves the theorem.
\begin{align*}
P(\gcd(z_{1},z_{2}) = 1) &= \left(1 - \frac{1}{2^2}\right) \prod_{p \equiv 1 \; \text{mod} \; 4} \left(1 - \frac{1}{p^2} \right)^2 \prod_{p \equiv 3 \; \text{mod} \; 4} \left(1 - \frac{1}{p^4}\right) \\
&= \frac{1}{\zeta(2)} \prod_{p \equiv 1 \; \text{mod} \; 4} \left(1 - \frac{1}{p^2}\right) \prod_{p \equiv 3 \; \text{mod} \; 4} \left(1 + \frac{1}{p^2}\right) = \frac{1}{\zeta(2)} \prod_{p \; \text{prime}}\left(1 - \frac{\chi(p)}{p^2}\right) \\
&= \left(\zeta(2) \sum_{n=1}^{\infty} \frac{\chi(n)}{n^2}\right)^{-1} = \frac{1}{\zeta(2)L(2,\chi)} = \frac{1}{\zeta_{\Q(i)}(2)}
\end{align*}
where $$\chi(n) =
\begin{cases}
1 \;\;\;\;\;\; \text{if } n \equiv 1 \mod{4} \\
-1 \;\;\;\; \text{if } n \equiv 3 \mod{4}
\end{cases}$$
is a character of $\Z[i]$, $L(2,\chi)$ is the Dirichlet L-function attached to $\chi$, and $\zeta_{\Q(i)}(s)$ is the Dedekind zeta function over the Gaussian integers.
\end{proof}

We note that this result generalizes as follows, where an adapted proof can be given by relating the properties of the multiplicity lattice to the group of units of the ring.

\begin{Theorem}
Let $K$ be a number field $K$ whose ring of integers is a principal ideal domain. Then,
$$P(\gcd(z_{1},...,z_{k})=1) = \frac{1}{\zeta_{K}(k)}$$
where $z_{1},...,z_{k} \in K$ are randomly chosen and $\zeta_{K}(s)$ is the Dedekind zeta function over $K$.
\end{Theorem}

\bibliography{numberbib}{}
\bibliographystyle{plain}

\end{document}